\documentclass{amsart}

\usepackage{latexsym,amsmath,amstext,mathrsfs,amsthm,amscd,amsopn,verbatim,amsrefs,tikz}
\tikzstyle{n}=[circle, draw, fill, minimum size=6, inner sep=0]
\tikzstyle{ext}=[circle, draw,  minimum size=3, inner sep=0]
\tikzstyle{int}=[circle, draw, fill, minimum size=3, inner sep=0]
\usetikzlibrary{matrix, arrows}

\usepackage{graphicx}
\usepackage[all]{xy}
\usepackage{amscd}
\usepackage{amssymb}
\usepackage{MnSymbol}
\usepackage[english, french]{babel}

\newtheorem{Thm}{Theorem}[section]

\newtheorem{Lem}[Thm]{Lemma}

\theoremstyle{remark}
\newtheorem{Rem}[Thm]{Remark}

\theoremstyle{definition}

\title[The explicit equivalence between...]{The explicit equivalence between the standard and the logarithmic star product for Lie algebras}
\author{C.~A.~Rossi}
\address{MPIM Bonn, Vivatsgasse 7, 53111 Bonn (Germany)}

\begin{document}



\maketitle

\selectlanguage{english}
\begin{abstract}
The purpose of this short note is to establish an explicit equivalence between the two star products $\star$ and $\star_{\log}$ on the symmetric algebra $\mathrm S(\mathfrak g)$ of a finite-dimensional Lie algebra $\mathfrak g$ over a field $\mathbb K\supset\mathbb C$ of characteristic $0$ associated with the standard angular propagator and the logarithmic one: the differential operator of infinite order with constant coefficients realizing the equivalence is related to the incarnation of the Grothendieck--Teichm\"uller group considered by Kontsevich in~\cite[Theorem 7]{K1}.
\end{abstract}


\selectlanguage{english}

\setcounter{tocdepth}{1}

\section{Introduction}\label{s-0}
For a general finite-dimensional Lie algebra $\mathfrak g$ over a field $\mathbb K\supset \mathbb C$, we consider the symmetric algebra $\mathrm S(\mathfrak g)$.

Deformation quantization {\em \`a la} Kontsevich~\cite{K} permits to endow $A=\mathrm S(\mathfrak g)$ with an associative, non-commutative product $\star$: the universal property of the Universal Enveloping Algebra (shortly, from now on, UEA) $\mathrm U(\mathfrak g)$ and a degree argument imply that there is an isomorphism of associative algebras $\mathcal I$ from $(A,\star)$ to $(\mathrm U(\mathfrak g),\cdot)$.
The algebra isomorphism $\mathcal I$ has been characterized explicitly in~\cite[Subsection 8.3]{K} and~\cite{Sh} as the composition of the Poincar\'e--Birkhoff--Witt (shortly, from now on, PBW) isomorphism (of vector spaces) with an invertible differential operator with constant coefficients and of infinite order associated to the well-known Duflo element $\sqrt{j(\bullet)}$ in the completed symmetric algebra $\widehat{\mathrm S}(\mathfrak g^*)$.
The case of nilpotent Lie algebras, where the aforementioned invertible differential operator equals simply the identity, has been considered in great detail in~\cite{Kath}, where the author has discussed the relationship between deformation quantization and the Baker--Campbell--Hausdorff formula.

In this short note, which takes inspiration from recent results~\cite{ALRT-1,ALRT-2} on the singular logarithmic propagator proposed by Kontsevich in~\cite[Subsection 4.1, F)]{K1}, we discuss the relationship between the star products $\star$ and $\star_{\log}$ on $A$, where $\star_{\log}$ is the star product associated with the logarithmic propagator.

The two star products $\star$ and $\star_{\log}$ on $A$ are obviously equivalent because of the fact that both associative algebras $(A,\star)$ and $(A,\star_{\log})$ are isomorphic to the Universal Enveloping Algebra $(\mathrm U(\mathfrak g),\cdot)$ of $\mathfrak g$ in virtue of degree arguments.

We produce here the explicit form of the equivalence between $\star$ and $\star_{\log}$ {\em via} a translation-invariant, invertible differential operator of infinite order on $A$ depending on the odd traces of the adjoint representation of $\mathfrak g$: more precisely, we provide an explicit formula relating $(A,\star_{\log})$ with $(\mathrm U(\mathfrak g),\cdot)$ {\em via} the PBW isomorphism, which we then compare with the previous one relating $(A,\star)$ with $(\mathrm U(\mathfrak g),\cdot)$.

The main result is a consequence of the logarithmic version of the formality result in presence of two branes from~\cite{CFFR} and the application discussed in~\cite{CFR} (``Deformation Quantization with generators and relations'').
Here a {\em caveat} is necessary: we do not prove here the general logarithmic formality in presence of two branes, which is quite technical and involved (deserving to it a separate and more detailed treating).
Here, we just discuss the main features and provide explicit formul\ae\ with a sketch of the main technicalities.





The present result provides a different insight to the incarnation of the Grothendieck--Teichm\"uller group in deformation quantization considered in~\cite[Theorem 7]{K1}.
Observe that, quite differently from~\cite{K1}, here odd traces of the adjoint representation of $\mathfrak g$ appear non-trivially, because we are not dealing with the Chevalley--Eilenberg cohomology of $\mathfrak g$ with values in $A$.

\subsection*{Acknowledgments} We thank J.~L\"offler for many useful discussions, for the careful reading of a first draft of the present note and for many useful suggestions.

\section{Notation and conventions}\label{s-1}
We consider a field $\mathbb K\supset \mathbb C$.

We denote by $\mathfrak g$ a finite-dimensional Lie algebra over $\mathbb K$ of dimension $d$; by $\{x_i\}$ we denote a $\mathbb K$-basis of $\mathfrak g$.
With $\mathfrak g$ we associate the (linear) Poisson variety $X=\mathfrak g^*$ over $\mathbb K$: the basis $\{x_i\}$ defines a set of global linear coordinates over $X$, and the Kirillov--Kostant Poisson bivector field $\pi$ on $X$ can be written as $\pi=f_{ij}^kx_k\partial_i\partial_j$, where we have omitted wedge product for the sake of simplicity, and $f_{ij}^k$ denote the structure constants of $\mathfrak g$ w.r.t.\ the basis $\{x_i\}$.

We denote by $\mathrm{ad}(\bullet)$ the adjoint representation of $\mathfrak g$ on itself; further, for $n\geq 1$, we denote by $c_n(\bullet)$ the element of $\mathrm S(\mathfrak g^*)$ defined {\em via} $c_n(x)=\mathrm{tr}_\mathfrak g(\mathrm{ad}(\bullet)^n)$.

Finally, $\zeta(\bullet)$ and $\Gamma(\bullet)$ denote the Riemann $\zeta$-function and the $\Gamma$-function respectively.

\section{An equivalence of star products incarnating the Grothendieck--Teichm\"uller group}\label{s-2}
For $\mathfrak g$ as in Section~\ref{s-1}, we consider the Poisson algebra $A=\mathbb K[X]=\mathrm S(\mathfrak g)$ endowed with the linear Kirillov--Kostant Poisson bivector field $\pi$.

Starting from $\pi$, we construct two distinct non-commutative, associative products $\star$ and $\star_{\log}$ on $A$, and we construct then an explicit equivalence between them: this equivalence is related to the Grothendieck--Teichm\"uller group $\mathrm{GRT}$ (or better, to its Lie algebra $\mathfrak{grt}$) in an explicit way, which points out the relationship between the logarithmic propagator and the $\mathrm{GRT}$-group.

\subsection{Explicit formul\ae\ for the products $\star$ and $\star_{\log}$}\label{ss-2-1}
Let $X=\mathbb K^d$ and $\{x_i\}$ a system of global coordinates on $X$, for $\mathbb K$ in Section~\ref{s-1}.

For a pair $(n,m)$ of non-negative integers, by $\mathcal G_{n,m}$ we denote the set of admissible graphs of type $(n,m)$: an element $\Gamma$ of $\mathcal G_{n,m}$ is a directed graph with $n$, resp.\ $m$, vertices of the first, resp.\ second type, such that $i)$ there is no directed edge departing from any vertex of the second type and $ii)$ $\Gamma$ admits whether multiple edges nor short loops ({\em i.e.} given two distinct vertices $v_i$, $i=1,2$, of $\Gamma$ there is at most one directed edge from $v_1$ to $v_2$ and there is no directed edge, whose endpoint coincides with the initial point). 
By $E(\Gamma)$ we denote the set of edges of $\Gamma$ in $\mathcal G_{n,m}$.

We denote by $C_{n,m}^+$ the configuration space of $n$ points in the complex upper half-plane $\mathbb H^+$ and $m$ ordered points on the real axis $\mathbb R$ {\em modulo} the componentwise action of rescalings and real translations: provided $2n+m-2\geq 0$, $C_{n,m}^+$ is a smooth, oriented manifold of dimension $2n+m-2$.
We denote by $\overline C_{n,m}^+$ a suitable compactification {\em \`a la} Fulton--MacPherson introduced in~\cite[Section 5]{K}: $\overline C_{n,m}^+$ is a compact, oriented, smooth manifold with corners of dimension $2n+m-2$.
We will be interested mostly in its boundary strata of codimension $1$. 

We denote by $\omega$, resp.\ $\omega_{\log}$ the closed, real-valued $1$-form
\[
\omega(z_1,z_2)=\frac{1}{2\pi}d\mathrm{arg}\!\left(\frac{z_1-z_2}{\overline z_1-z_2}\right),\ \text{resp.}\ \omega_{\log}(z_1,z_2)=\frac{1}{2\pi i}d\log\!\left(\frac{z_1-z_2}{\overline z_1-z_2}\right),\ (z_1,z_2)\in (\mathbb H^+\sqcup \mathbb R)^2,\ z_1\neq z_2,
\]
where $\mathrm{arg}(\bullet)$ denotes the $[0,2\pi)$-valued argument function on $\mathbb C\smallsetminus\{0\}$ such that $\mathrm{arg}(i)=\pi/2$, and $\log(\bullet)$ denotes the corresponding logarithm function, such that $\log(z)=\ln(|z|)+i\mathrm{arg}(z)$.
 
The $1$-form $\omega$ extends to a smooth, closed $1$-form on $\overline C_{2,0}^+$, such that $i)$ when the two arguments approach to each other in $\mathbb H^+$, $\omega$ equals the normalized volume form $d\varphi$ on $S^1$ and $ii)$ when the first argument approaches $\mathbb R$, $\omega$ vanishes.

On the other hand, $\omega_{\log}$ extends smoothly to all boundary strata of $\overline C_{2,0}^+$ ({\em e.g.} through a direct computation, one sees that $\omega_{\log}$ vanishes, when its first argument approaches $\mathbb R$ and coincides with $\omega$ when the second argument approaches $\mathbb R$) except the one corresponding to the collapse of its two arguments in $\mathbb H^+$, where it has a complex pole of order $1$.

The standard propagator $\omega$ has been introduced and discussed in~\cite[Subsection 6.2]{K}; the logarithmic propagator $\omega_{\log}$ has been first introduced in~\cite[Subsection 4.1, F)]{K1}.

We introduce $T_\mathrm{poly}(X)=A[\theta_1,\dots,\theta_d]$, $A=C^\infty(X)$, where $\{\theta_i\}$ denotes a set of graded variables of degree $1$, which commute with $A$ and anticommute with each other (one may think of $\theta_i$ as $\partial_i$ with a shifted degree).
We further consider the well-defined linear endomorphism $\tau$ of $T_\mathrm{poly}(X)^{\otimes 2}$ of degree $-1$ defined {\em via}
\[
\tau=\partial_{\theta_i}\otimes\partial_{x_i},
\]
where of course summation over repeated indices is understood.
We set $\omega_\tau=\omega\otimes \tau$ and similarly for $\omega_\tau^{\log}$.

With $\Gamma$ in $\mathcal G_{n,m}$ such that $|E(\Gamma)|=2n+m-2$, $\gamma_i$, $i=1,\dots,n$, elements of $T_\mathrm{poly}(X)$ and $a_j$, $j=1,\dots,m$, elements of $A$, we associate two maps $\mathcal U_\Gamma$, $\mathcal U_\Gamma^{\log}$ {\em via}
\[
\left(\mathcal U_\Gamma(\gamma_1,\dots,\gamma_n)\right)(a_1\otimes\cdots\otimes a_m)=\mu_{m+n}\left(\int_{C_{n,m}^+}\omega_{\tau,\Gamma}\left(\gamma_1\otimes\cdots\otimes\gamma_n\otimes a_1\otimes\cdots\otimes a_m\right)\right),\ \omega_{\tau,\Gamma}=\prod_{e\in E(\Gamma)}\omega_{\tau,e},\ \omega_{\tau,e}=\pi_e^*(\omega)\otimes\tau_e,
\]
$\tau_e$ being the graded endomorphism of $T_\mathrm{poly}(X)^{\otimes(m+n)}$ which acts as $\tau$ on the two factors of $T_\mathrm{poly}(X)$ corresponding to the initial and final point of the edge $e$, and $\mu_{m+n}$ denotes the multiplication map from $T_\mathrm{poly}(X)^{m+n}$ to $T_\mathrm{poly}(X)$, followed by the natural projection from $T_\mathrm{poly}(X)$ onto $A$ by setting $\theta_i=0$, $i=1,\dots,d$; $\mathcal U_\Gamma^{\log}$ is defined as in the previous formula by replacing overall $\omega$ by $\omega_{\log}$.

We may re-write both $\mathcal U_\Gamma$ and $\mathcal U_\Gamma^{\log}$ splitting the form-part and the polydifferential operator part as
\[
\begin{aligned}
\left(\mathcal U_\Gamma(\gamma_1,\dots,\gamma_n)\right)(a_1\otimes\cdots\otimes a_m)&=\varpi_\Gamma(\mathcal B_\Gamma(\gamma_1,\dots,\gamma_n))(a_1,\dots,a_m),\ \varpi_\Gamma=\int_{C_{n,m}^+}\omega_\Gamma,\\
\left(\mathcal U_\Gamma^{\log}(\gamma_1,\dots,\gamma_n)\right)(a_1\otimes\cdots\otimes a_m)&=\varpi_\Gamma^{\log}(\mathcal B_\Gamma(\gamma_1,\dots,\gamma_n))(a_1,\dots,a_m),\ \varpi_\Gamma^{\log}=\int_{C_{n,m}^+}\omega_\Gamma^{\log}.
\end{aligned}
\]
The polydifferential parts of $\mathcal U_\Gamma$ and $\mathcal U_\Gamma^{\log}$ are equal, while the corresponding integral weights $\varpi_\Gamma$ and $\varpi_\Gamma^{\log}$ are different.
\begin{Thm}\label{t-star}
For a Poisson bivector field $\pi$ on $X$ and a formal parameter $\hbar$, the formul\ae
\begin{equation}\label{eq-star}
f_1\star_{\hbar} f_2=\sum_{n\geq 0}\frac{\hbar^n}{n!}\sum_{\Gamma\in\mathcal G_{n,2}}(\mathcal U_\Gamma(\underset{n}{\underbrace{\pi,\dots,\pi}}))(f_1,f_2),\ f_1\star_{\log,\hbar} f_2=\sum_{n\geq 0}\frac{\hbar^n}{n!}\sum_{\Gamma\in\mathcal G_{n,2}}(\mathcal U_\Gamma^{\log}(\underset{n}{\underbrace{\pi,\dots,\pi}}))(f_1,f_2),\ f_i\in A,\ i=1,2,
\end{equation}
define $\mathbb K_{\hbar}=\mathbb K[\!\![\hbar]\!\!]$-linear, associative products on $A_{\hbar}=A[\!\![\hbar]\!\!]$. 
\end{Thm}
The first expression in~\eqref{eq-star} has been proved to be well-defined ({\em i.e.} all integrals converge) and to yield an associative product as a corollary of the Formality Theorem~\cite[Theorem 6.4]{K}.

On the other hand, a Formality Theorem in presence of $\omega^{\log}$ has been proved recently in~\cite{ALRT-1,ALRT-2}, from which follows that the second expression in~\eqref{eq-star} is a well-defined, associative product.
Appendix B contains a sketch of the technical arguments explained in detail in~\cite{ALRT-1,ALRT-2}. 

\subsection{Relationship between $\star$, $\star_{\log}$ and the UEA of $\mathfrak g$}\label{ss-2-2}
We now restrict our attention to $X=\mathfrak g^*$, for $\mathfrak g$ as in Section~\ref{s-1}, endowed with the Kirillov--Kostant Poisson bivector $\pi$.

Degree reasons imply that the products in~\eqref{eq-star} restrict to $A_{\hbar}$, $A=\mathrm S(\mathfrak g)$, and that the $\hbar$-dependence of both of them is in fact polynomial: therefore, we may safely set $\hbar=1$, and we use the short-hand notation $\star$ and $\star_{\log}$ on $A$.

With $\mathfrak g$, we associate its UEA (short for Universal Enveloping Algebra) $(\mathrm U(\mathfrak g),\cdot)$; we denote by $\mathrm{PBW}$ the symmetrization isomorphism (of vector spaces) from $A$ to $\mathrm U(\mathfrak g)$.
\begin{Thm}\label{t-star-UEA}
For $\mathfrak g$ as in Section~\ref{s-1}, there exist isomorphisms of associative algebras $\mathcal I$ and $\mathcal I_{\log}$ from $(A,\star)$ and $(A,\star_{\log})$ respectively to $(\mathrm U(\mathfrak g),\cdot)$, which are explicitly given by
\begin{equation}\label{eq-star-UEA}
\mathcal I=\mathrm{PBW}\circ \sqrt{j(\bullet)},\quad \mathcal I_{\log}=\mathrm{PBW}\circ j_\Gamma(\bullet),
\end{equation}
where $\sqrt{j(\bullet)}$ and $j_\Gamma(\bullet)$ are elements of $\widehat{\mathrm S}(\mathfrak g^*)$ defined {\em via}
\begin{align}
\label{eq-star-UEA-1}\sqrt{j(x)}&=\sqrt{\mathrm{det}_\mathfrak g\!\left(\frac{1-e^{-\mathrm{ad}(x)}}{\mathrm{ad}(x)}\right)}=\exp\!\left(-\frac{1}4c_1(x)+\sum_{n\geq 1}\frac{\zeta(2n)}{(2n)(2\pi i)^{2n}}c_{2n}(x)\right),\\
\label{eq-star-UEA-2}j_\Gamma(x)&=\exp\!\left(-\frac{1}4 c_1(x)+\sum_{n\geq 2}\frac{\zeta(n)}{n(2\pi i)^n}c_n(x)\right)=\sqrt{j(x)}\ \exp\!\left(\sum_{n\geq 1}\frac{\zeta(2n+1)}{(2n+1)(2\pi i)^{2n+1}}c_{2n+1}(x)\right),\ x\in\mathfrak g,
\end{align}
where both elements of the completed symmetric algebra $\widehat{\mathrm S}(\mathfrak g^*)$ are regarded as invertible differential operators with constant coefficients and of infinite order on $A$. 

(We will comment at the end of the proof on the (improperly) adopted notation for both expressions~\eqref{eq-star-UEA-1} and~\eqref{eq-star-UEA-2}.)
\end{Thm}
\begin{proof}
The identity on the left-hand side of~\eqref{eq-star-UEA} has been proved in~\cite[Subsection 8.3]{K} by means of the compatibility between cup products; a different proof has been presented in~\cite[Subsection 3.2]{CFR}.
We will adopt the strategy proposed in~\cite[Subsection 3.2]{CFR} to prove the identity on the right-hand side, to which we refer for more details. 

Let us momentarily re-introduce the formal parameter $\hbar$, and consider the corresponding $\hbar$-formal Poisson bivector $\hbar\pi$.

To $\mathfrak g$, we may attach two natural quadratic algebras, $A$ and $B=\wedge(\mathfrak g^*)$: observe that, in the present framework, $A$ is concentrated in degree $0$, while $B$ is non-negatively graded.
It is well-known that $A$ and $B$ are Koszul algebras, and moreover they are Koszul dual to each other.

We consider then the (graded) algebras $A_{\hbar}$, $B_{\hbar}$ over the ring $\mathbb K[\!\![\hbar]\!\!]$.
With the formal Poisson structure $\hbar\pi$, we associate the product $\star_{\log,\hbar}$ {\em via} the formality quasi-isomorphism $\mathcal U^{\log}$.

On the other hand, we may consider the $\hbar$-formal Fourier dual quadratic vector field $\hbar\widehat\pi=\hbar f_{ij}^k\theta_i\theta_j\partial_{\theta_k}$, borrowing previous notation for the graded basis $\{\theta_i\}$ of $B$, on $B_{\hbar}$ ($\hbar\widehat\pi$ is the $\hbar$-shifted Chevalley--Eilenberg differential $d_{\hbar}$ on $B_{\hbar}$).
Thus, the triple $(B_{\hbar},d_{\hbar},\wedge)$ is a dg algebra over $\mathbb K[\!\![\hbar]\!\!]$: the graded formality quasi-isomorphism $\mathcal V$ in~\cite[Appendix A]{CF} admits a logarithmic version $\mathcal V^{\log}$ simply by replacing everywhere $\omega$ by $\omega_{\log}$, and the MC element $\mathcal V^{\log}(\hbar\widehat\pi)$ endows $B_{\hbar}$ with the $A_\infty$-structure over $\mathbb K[\!\![\hbar]\!\!]$ given by $\wedge+\mathcal V^{\log}(\hbar\widehat\pi)$.
Degree arguments and the fact that the logarithmic integral weight $\varpi^{\log}_\Gamma$ associated to the graph $\Gamma$ in depicted in Figure~\ref{fig-1}, $i)$, is trivial yield that the only non-trivial Taylor components of the aforementioned $A_\infty$-structure are $d_{\hbar}$ and $\wedge$, thus deformation quantization produces out of the graded commutative algebra $(B,\wedge)$ the $\hbar$-shifted Chevalley--Eilenberg complex $(B_{\hbar},d_{\hbar},\wedge)$.
The computation of the said weight $\varpi_\Gamma^{\log}$ has been performed in~\cite[Lemma 6.8]{ALRT-2}: it follows from Stokes' Theorem~\ref{t-stokes-reg}, Appendix B.

We refer to Appendix A for a very quick review of the needed $A_\infty$-structures in the forthcoming discussion.

We regard $A$ and $B$ as unital $A_\infty$-algebras: in~\cite[Subsection 6.2]{CFFR} a non-trivial $A_\infty$-$A$-$B$-bimodule structure on $K=\mathbb K$ has been explicitly constructed, which restricts to the standard augmentation $A$- and $B$-module structure and such that $\mathrm L_A:A\to \underline{\mathrm{End}}_{B^+}(K)$ is an $A_\infty$-quasi-isomorphism.

Observe that we may also consider $A_\infty$-algebras and $A_\infty$-bimodules over the ring $K[\!\![\hbar]\!\!]$, and all previous definitions and constructions apply to this setting as well.
In particular, we may regard $(A_{\hbar},\star_{\log,\hbar})$ and $(B_{\hbar},d_{\hbar},\wedge)$ are $A_\infty$-algebras over $\mathbb K[\!\![\hbar]\!\!]$.
\begin{Lem}\label{l-A_infty}
There exists an $\hbar$-formal flat deformation $K_{\hbar}=\mathbb K[\!\![\hbar]\!\!]$ of the $A_\infty$-$A$-$B$-bimodule $K$ as an $A_\infty$-$(A_{\hbar},\star_{\log,\hbar})$-$(B_{\hbar},d_{\hbar},\wedge)$-bimodule.
\end{Lem}
\begin{proof}[Sketch of proof of Lemma~\ref{l-A_infty}]
Let us consider the first quadrant $Q^{+,+}$ in $\mathbb C$, with which we associate the configuration space $C_{2,0,0}^+$ of two distinct points in $Q^{+,+}$ {\em modulo} rescalings.
We define on $C_{2,0,0}^+$ the closed, complex-valued $1$-form
\begin{equation}\label{eq-4-log}
\omega^{+,-}_{\log}(z_1,z_2)=\frac{1}{2\pi i}d\log\!\left(\frac{z_1-z_2}{\overline z_1-z_2}\frac{\overline z_1+z_2}{z_1+z_2}\right)-\frac{1}{\pi i}d\log\!\left(\left|\frac{\overline z_1+z_2}{z_1+z_2}\right|\right).
\end{equation}

We denote by $\overline C_{2,0,0}^+$ the compactified version {\em \`a la} Fulton--MacPherson of $C_{2,0,0}^+$, see~\cite[Subsection 3.1]{CRT} for a complete description thereof and of its boundary stratification.
The $1$-form~\eqref{eq-4-log} has the following properties:
\begin{itemize}
\item[$i)$] when both arguments of $\omega_{\log}^{+,-}$ approach $\mathbb R^+$, resp.\ $i\mathbb R^+$, $\omega^{+,-}_{\log}=\omega^-_{\log}$, resp.\ $\omega^{+,-}_{\log}=\omega^+_{\log}$, where $\omega_{\log}^+=\omega_{\log}$ and $\omega_{\log}^-=\sigma^*(\omega_{\log})$, $\sigma$ being the involution of $C_{2,0}^+$ given by $(z_1,z_2)\mapsto (z_2,z_1)$;
\item[$ii)$] $\omega^{+,-}_{\log}$ vanishes, when its initial, resp.\ final point, approaches $i\mathbb R^+$, resp. $\mathbb R^+$, or the origin;
\item[$iii)$] $\omega^{+,-}_{\log}$ has a simple pole of order $1$ along the boundary stratum $S^1\times C_{1,0,0}^+$ of $\overline C_{2,0,0}^+$ corresponding to the collapse of its two arguments to a single point in $Q^{+,+}$, and the $S^1$-piece of the corresponding regularization (see~\cite[Subsubsection 2.1.1]{ALRT-2} or Appendix B) equals the normalized volume form of $S^1$.
There is also a $C^+_{1,0,0}$-piece in the regular part, whose presence justifies the fact that the logarithmic counterpart of the formality result for two branes requires admissible graphs with short loops, see also later on.
\end{itemize}
Properties $i)$-$iii)$ can be checked by direct computations using local coordinates as in the proof of~\cite[Lemma 5.4]{CFFR}, with due modifications because of the pole.

The explicit formul\ae\ for the $A_\infty$-bimodule structure $d_{K_{\hbar}}^{k,l}$ on $K_{\hbar}$ can be obtained from the ones for the corresponding $A_\infty$-structure on $K_{\hbar}$ constructed in~\cite[Section 7]{CFFR} by replacing $\omega^{+,-}$ by its logarithmic counterpart~\eqref{eq-4-log}: also for later computations, we frequently and implicitly refer to~\cite[Section 7]{CFFR}.

The convergence of the logarithmic integral weights appearing in the $A_\infty$-bimodule structure, as well as the $A_\infty$-property itself, will be shown in a forthcoming paper in their full generality: still, we refer to the arguments sketched in the Appendix B for a proof of the convergence of the logarithmic integral weights (see also~\cite[Proposition 4.2]{ALRT-2}), while the $L_\infty$-property follows by means of Stokes' Theorem~\ref{t-stokes-reg} (see~\cite[Theorem 1.8]{ALRT-1}), with some due modifications which arise from the regular parts of the $4$-colored logarithmic propagators involved.

Observe that, in view of the properties of~\eqref{eq-4-log}, if we set $\hbar=0$, we recover the $A_\infty$-$A$-$B$-bimodule structure on $K$ from~\cite{CFR}.
\end{proof}
In particular,~\cite[Proposition 2.4, Lemma 2.5 and Theorem 2.7]{CFR} are valid in their full generality in the logarithmic framework as well: in particular, there is an algebra isomorphism 
\begin{equation}\label{eq-adam-def}
\mathrm L_{A_{\hbar}}^{1,\log}:(A_{\hbar},\star_{\log,\hbar})\to \mathrm T(\mathfrak g)/\left(\mathrm T(\mathfrak g)\otimes \langle x_i\otimes x_j-x_j\otimes x_i-\hbar[x_i,x_j]:\ i,i=1,\dots,d\rangle\otimes \mathrm T(\mathfrak g)\right)[\!\![\hbar]\!\!]=(\mathrm U_{\hbar}(\mathfrak g),\cdot),
\end{equation}
where $\mathrm U_{\hbar}(\mathfrak g)$ is the UEA of the $\hbar$-shifted Lie algebra $\mathfrak g_{\hbar}=\mathfrak g$, with Lie bracket $\hbar[\bullet,\bullet]$.

The algebra isomorphism~\eqref{eq-adam-def} is a particular case of the logarithmic version of Shoikhet's conjecture~\cite{Sh1} about deformation quantization with generators and relations. 

Again, degree reasons imply that we may safely set $\hbar=1$ in~\eqref{eq-adam-def}, which in particular implies that $\mathrm L_A^{1,\log}$ yields an algebra isomorphism from $(A,\star_{\log})$ to $(\mathrm U(\mathfrak g),\cdot)$.
\begin{Lem}\label{l-I_log}
The algebra isomorphism $\mathrm L_A^{1,\log}$ can be computed explicitly and equals $\mathcal I_{\log}$.
\end{Lem}
\begin{proof}[Sketch of the proof of Lemma~\ref{l-I_log}]
The quasi-isomorphism $\mathrm L_A^{1,\log}$ is explicitly given by the formula
\[
\mathrm L_A^{1,\log}(a_1)^m(1,b_1,\dots,b_m)=d_K^{1,m}(a_1,1,b_1,\dots,b_m),
\] 
where one must think of $d_K^{1,m}$ as of $d_{K_{\hbar}}^{1,m}$ for $\hbar=1$, and $a_1$ in $A$, and $b_i$ in $B$, $i=1,\dots,m$.

Using the graphical definition of the deformed $A_\infty$-$A$-$B$-bimodule structure specified by the Taylor components $d_K^{k,l}$ in~\cite[Subsections 6.2, 7.1 and 7.2]{CFFR}, we consider an admissible graph $\Gamma$ of type $(n,1,m)$ with $n$ vertices of the first type ({\em i.e.} in $Q^{+,+}$), $1$ vertex of the second type on $i\mathbb R^+$ and $m$ ordered vertices of the second type on $\mathbb R$; $\Gamma$ may have short loops at vertices of the first type, no edge may depart from the vertex on $i\mathbb R^+$ and no edge may arrive at a vertex on $\mathbb R$.
From any vertex of the first type depart exactly two directed edges, whence $|E(\Gamma)|=2n+m$.
\begin{figure}
\centering
\begin{tikzpicture}[>=latex]
\tikzstyle{k-int}=[draw,fill=gray!40,circle,inner sep=0pt,minimum size=2.5mm]
\tikzstyle{n-int}=[draw,fill=black,circle,inner sep=0pt,minimum size=2.5mm]
\tikzstyle{ext}=[draw,fill=white,circle,inner sep=0pt,minimum size=2.5mm]
\tikzstyle{spec}=[draw,rectangle,inner sep=0pt,minimum size=3mm]

\begin{scope}
\node[n-int] (1) at (-0.7,1.5) {};
\node[n-int] (2) at (.7,1.5) {};
\draw[thick,->] (1) [out=45,in=135] to (2);
\draw[thick,->] (2) [out=225,in=-45] to (1);
\end{scope}
\node at (0,-.5) {$i)$};
\begin{scope}[scale=0.75,shift={(6,0)}]
\draw (0,0) to (4,0);
\draw (0,0) to (0,4);
\node[ext] (v1) at (0,2) {};
\node[ext] (h1) at (1,0) {};
\node[ext] (h2) at (2,0) {};
\node[ext] (h3) at (3,0) {};
\begin{scope}[scale=0.75,shift={(3,4.5)}]
\node[n-int] (w1) at (30:1.5) {};
\node[n-int] (w2) at (150:1.5) {};
\node[n-int] (w3) at (270:1.5) {};
\draw[thick,->] (w1) to (w2);
\draw[thick,->] (w2) to (w3);
\draw[thick,->] (w3) to (w1);
\end{scope}
\draw[thick,->] (w1) to (v1);
\draw[thick,->] (w2) to (v1);
\draw[thick,->] (w3) to (v1);
\draw[thick,->] (h1) [out=90,in=0] to (v1);
\draw[thick,->] (h2) [out=90,in=0] to (v1);
\draw[thick,->] (h3) [out=90,in=0] to (v1);
\end{scope}
\node at (5.5,-.5) {$ii)$};
\end{tikzpicture}
\caption{\label{fig-1} \selectlanguage{english} $i)$ The unique admissible graph of type $(2,0)$ and two directed edges; $ii)$ an admissible graph $\Gamma$ of type $(3,1,3)$ contributing to $\mathrm L_A^{1,\log}$.
}
\end{figure}
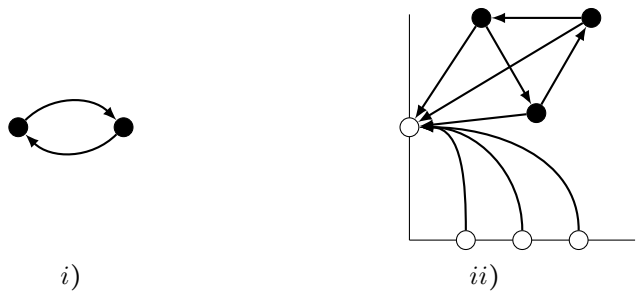
The multidifferential operator associated to an admissible graph $\Gamma$ of type $(n,1,m)$ as before is analogous to the one appearing in the construction of~\eqref{eq-star}: $b_i$ in $B$, $i=1,\dots,m$, is regarded as a polyderivation with constant coefficients on $A$, while a short loop corresponds to the divergence operator on $T_\mathrm{poly}(X)$ w.r.t.\ the constant volume form on $X$.

The corresponding integral weight $\varpi^{\log,+,-}_\Gamma$ is obtained by associating with an edge between two distinct vertices, resp.\ a short loop at a vertex of the first type, the closed $1$-form~\eqref{eq-4-log} on $C_{2,0,0}^+$, resp.\ the exact $1$-form $d\mathrm{arg}(\bullet)/(4\pi)$ on $C_{1,0,0}^+$: then one integrates the corresponding closed form $\omega_\Gamma^{\log,+,-}$ of degree $2n+m$ over $C_{n,1.m}^+$.
The form $\omega_\Gamma^{\log,+,-}$ extends to a complex-valued, real analytic closed form of top degree on $\overline C_{n,1,m}^+$, whence $\varpi_\Gamma^{\log,+,-}$ converges.

The fact that $\pi$ is a linear bivector field implies that a vertex of the first type of $\Gamma$ can be the endpoint of at most one edge.
Moreover, the degree of $\omega_\Gamma^{\log,+,-}$ equals $2n+\sum_{j=1}^m|b_j|$, where $|b_j|$ denotes the degree of $b_j$ as a constant polyderivation on $A$, whence $\sum_{j=1}^m|b_j|=m$. 
Dimensional reasons for $\varpi_\Gamma^{\log,+,-}$ imply that $|b_j|=1$, $j=1,\dots,m$: namely, if $|b_i|\geq 2$, for $i=1,\dots,m$, then $|b_j|=0$, for $i\neq j$, whence $\Gamma$ would have a $0$-valent vertex of the second type.
In other words, there would exist a point on a $1$-dimensional manifold, along which no form is integrated, yielding a trivial integral weight.

Furthermore, from any vertex of the first type may depart at most one edge to the only vertex on $i\mathbb R^+$ (otherwise, $\omega_\Gamma^{\log,+,-}$ would contain a square of~\eqref{eq-4-log}). 
Let now $p$ denote the number of edges from vertices of the first type hitting the only vertex on $i\mathbb R^+$: then, obviously, $p\leq n$.
With any edge or short loop of $\Gamma$ is associated a derivative w.r.t.\ $\{x_i\}$: the degree in $A$ of the multidifferential operator equals $n-(2n-p)-j=-n-j+p\geq 0$, where $0\leq j\leq m$ denotes the number of edges starting from vertices on $\mathbb R$ and hitting vertices of the first type, thus $p\geq n+j$.
It follows immediately that $j=0$ and $p=n$, {\em i.e.} edges from vertices on $\mathbb R$ may arrive only at the unique vertex on $i\mathbb R$, and from any vertex of the first type exactly one edge hits the said vertex on $i\mathbb R^+$, while the other edge may hit any vertex of the first type (including the initial point itself of the given edge).

Thus, a general admissible graph $\Gamma$ of type $(n,1,m)$ is the disjoint union of wheel-like graphs with spokes pointing towards the unique vertex on $i\mathbb R^+$ and a graph with single edges starting from ordered vertices on $\mathbb R$ and hitting the unique vertex on $i\mathbb R^+$, as in Figure~\ref{fig-1}, $ii)$.

The behavior of~\eqref{eq-4-log} along the boundary strata of $\overline C_{2,0,0}^+$ implies that the multidifferential operators associated with admissible graphs with no vertices of the first type contribute to the symmetrization isomorphism $\mathrm{PBW}$ from $A$ to $\mathrm U(\mathfrak g)$, see also~\cite[Subsection 4.2]{CFR}.

On the other hand, the wheel-like graphs sum up to yield exactly the invertible differential operator of infinite order and with constant coefficients specified by $j_\Gamma(\bullet)$ in $\widehat{\mathrm S}(\mathfrak g^*)$, once we have computed the logarithmic integral weights of the wheel-like graphs.

For this, we use the strategy adopted in~\cite[Appendix B]{W}, where we replace $\omega^{+,-}$ by its logarithmic counterpart $\omega_{\log}^{+,-}$.

More precisely, the discussion in the first part of Appendix B implies that Stokes' Theorem~\ref{t-stokes-reg}, Appendix B, applies to differential forms associated with a wheel-like graphs with $n+1$ vertices as before.
The boundary conditions for $\omega^{+,-}_{\log}$ and the regularization morphism imply the graphical relation among logarithmic integral weights depicted in Figure~\ref{fig-2} for $n=3$.
Observe that we have adopted the lazy convention for the signs: however, signs behave exactly as in~\cite[Appendix B]{W}, because of regularization morphism does not alter signs.
\begin{figure}
\centering
\begin{tikzpicture}[>=latex]
\tikzstyle{k-int}=[draw,fill=gray!40,circle,inner sep=0pt,minimum size=2.5mm]
\tikzstyle{n-int}=[draw,fill=black,circle,inner sep=0pt,minimum size=2.5mm]
\tikzstyle{ext}=[draw,fill=white,circle,inner sep=0pt,minimum size=2.5mm]
\tikzstyle{spec}=[draw,rectangle,inner sep=0pt,minimum size=3mm]

\begin{scope}[scale=0.75]
\draw (0,0) to (3,0);
\draw (0,0) to (0,3);
\node[ext] (v) at (0,1) {};
\begin{scope}[shift={(35:3)}]
\node[n-int] (w1) at (0:1) {};
\node[n-int] (w2) at (120:1) {};
\node[n-int] (w3) at (240:1) {};
\draw[thick,->] (w1) to (w2);
\draw[thick,->] (w2) to (w3);
\draw[thick,->] (w3) to (w1);
\end{scope}
\draw[thick,->] (w1) to (v);
\draw[thick,->] (w2) to (v);
\draw[thick,->] (w3) to (v);
\end{scope}
\node at (3.25,1.25) {$+$};
\begin{scope}[scale=0.75,shift={(6,0)}]
\node[n-int] (c) at (0,0) {};
\begin{scope}[shift={(0,2)}]
\node[n-int] (w1) at (90:1) {};
\node[n-int] (w2) at (210:1) {};
\node[n-int] (w3) at (330:1) {};
\draw[thick,->] (w1) to (w2);
\draw[thick,->] (w2) to (w3);
\draw[thick,->] (w3) to (w1);
\end{scope}
\draw[thick,->] (w1) to (c);
\draw[thick,->] (w2) to (c);
\draw[thick,->] (w3) to (c);
\end{scope}
\node at (6,1.25) {$+$};
\begin{scope}[scale=0.75,shift={(10,0)}]
\node[n-int] (c) at (0,0) {};
\begin{scope}[shift={(0,2)}]
\node[n-int] (w1) at (90:1) {};
\node[n-int] (w2) at (210:1) {};
\node[n-int] (w3) at (330:1) {};
\draw[thick,->] (w2) to (w1);
\draw[thick,->] (w3) to (w2);
\draw[thick,->] (w1) to (w3);
\end{scope}
\draw[thick,->] (c) to (w1);
\draw[thick,->] (c) to (w2);
\draw[thick,->] (c) to (w3);
\end{scope}
\node at (9,1.25) {$=0$};
\end{tikzpicture}
\caption{\label{fig-2} \selectlanguage{english} Graphical representation of the relation between weights of wheel-like graphs.
}
\end{figure}
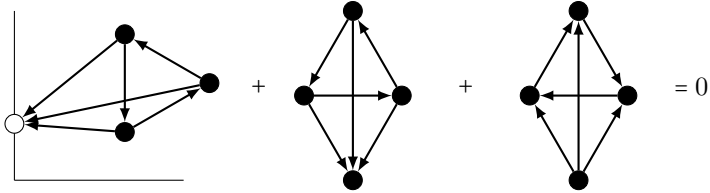

The second and third integral weights have to be understood w.r.t.\ the logarithmic propagator $\omega_{\log}$ and on configuration spaces $\overline C_{n+1,0}^+$ of distinct points in $\mathbb H^+$.

The integral weight of a wheel-like graph as in the third term has been computed explicitly in~\cite[Appendix A]{Merk}; it remains to prove that the integral weight of a wheel-like graph as in the second term vanishes.
This is, in turn, a consequence of a more general Vanishing Lemma for logarithmic integral weights.
\begin{Lem}\label{l-van}
Let $\Gamma$ be an admissible graph of type $(n,m)$ with $|E(\Gamma)|=2n+m-2$, $n\geq 2$, admitting a vertex of the first type which is the starting point of no edge.

Then, the corresponding logarithmic integral weight $\varpi_\Gamma^{\log}$ vanishes.
\end{Lem}
\begin{proof}
We may safely assume, because of the standard dimensional argument, that the vertex $v$ of the first type from which no edge departs is the endpoint of $l\geq 2$ edges.

Observe first that $\omega_{\log}$ depends holomorphically on its second argument.
Further, since $n\geq 2$, we may {\em e.g.} fix to $i$ the coordinate corresponding to some vertex $v_1\neq v$ of the first type; let us denote by $z_v$ the coordinate corresponding to the vertex $v$. 

Then, the differential form $\omega^{\log}_\Gamma$ vanishes, because it depends only holomorphically on $z_v$ and there is no non-trivial form of top degree on $\overline C_{n,m}^+$ which is holomorphic in one complex coordinate.
\end{proof}
Lemma~\ref{l-van} obviously applies to wheel-like graph as in the second term of Figure~\ref{fig-2}, and, because of the previous computations, yields the claim of Lemma~\ref{l-I_log}.
\end{proof}
Lemma~\ref{l-I_log} is proved, and this, in turn, yields the desired claim about the explicit expression for the algebra isomorphism relating the logarithmic star product $\star_{\log}$ on $A$ and the associative product on $\mathrm U(\mathfrak g)$.
\end{proof}
As remarked before the proof of Lemma~\ref{l-van}, the computation of the weights corresponding to the contributions coming from wheel-like graphs is found in~\cite[Appendix A]{Merk}, whence we deduce both expressions~\eqref{eq-star-UEA-1} and~\eqref{eq-star-UEA-2} in the form of exponentials of convergent power series, whose coefficients depend on $\zeta(\bullet)$.
Furthermore, it is well-known that $c_1$ is a derivation of both $\star$ and $\star_{\log}$ on $A$.
Therefore, by suitably changing the first coefficient, the invertible differential operators in~\eqref{eq-star-UEA-1} and~\eqref{eq-star-UEA-2} are obtained (up to the coefficient of $c_1$, which may be chosen freely) from the functions
\[
\begin{aligned}
\sqrt{j(z)}&=\sqrt{\frac{1-e^{-z}}z}=\exp\!\left(-\frac{1}4 z+\sum_{n\geq 1}\frac{\zeta(2n)}{(2n)(2\pi i)^n}z^{2n}\right)=\exp\!\left(-\frac{1}4 z+\sum_{n\geq 1}\frac{B_{2n}}{(4n)(2n)!}z^{2n}\right),\\
\frac{1}{\Gamma\left(1+\frac{z}{2\pi i}\right)}&=\exp\!\left(\gamma z+\sum_{n\geq 1}\frac{\zeta(n)}{n(2\pi i)^n}z^n\right)=\sqrt{j(z)}\ \exp\!\left(\sum_{n\geq 1}\frac{\zeta(2n+1)}{(2n+1)(2\pi i)^{2n+1}}z^{2n+1}\right),
\end{aligned}
\]
where $\gamma$ denotes the Euler--Mascheroni constant, and $B_n$, $n\geq 2$, denotes the $n$-th Bernoulli number.
Observe that the constant $\gamma$ appears mainly for aesthetical reasons.

Therefore, a bit improperly, we may use the reciprocal $\Gamma$-function with shifted argument to construct the isomorphism $\mathcal I_{\log}$: in fact, up to the term of first order (whose coefficient may be chosen freely because $c_1$ is a derivation for both products $\star$ and $\star_{\log}$), the exponential of the power series in~\eqref{eq-star-UEA-2} coincides with the function first considered in~\cite[Subsection 4.6]{K1} in a discussion about incarnations of the GRT group.
The very same expression has been re-discovered in~\cite[Subsection 4.9]{Merk} in the framework of exotic $L_\infty$-automorphisms of $T_\mathrm{poly}(X)$ and their connection with the GRT group.

As an immediate corollary of Theorem~\ref{t-star-UEA}, for a Lie algebra $\mathfrak g$ as in Section~\ref{s-1}, the star products in~\eqref{eq-star} on $A$ are equivalent w.r.t.\ the invertible differential operator with constant coefficients and of infinite order associated with the element of $\widehat{\mathrm S}(\mathfrak g)$ given by
\[
\exp\!\left(\sum_{n\geq 1}\frac{\zeta(2n+1)}{(2n+1)(2\pi i)^{2n+1}}c_{2n+1}(x)\right),\ x\in\mathfrak g.
\]

Let us remark that a more conceptual approach to the Lie algebra $\mathfrak{grt}$ in the framework of deformation quantization can be found in the fundamental paper~\cite{W1}, in particular in~\cite[Subsections 7.4, 7.5]{W1}, where modifications of the Duflo element {\em via} elements of $\mathfrak{grt}$ have been discussed in details.
\begin{Rem}\label{r-van}
We finally observe that Lemma~\ref{l-van} generalizes to the logarithmic framework the results of~\cite{Sh}: we only point out that the our Vanishing Lemma applies to a wider variety of situations.
\end{Rem}

\section*{Appendix A: a very quick review of $A_\infty$-structures}
Let $C$ be a graded vector space over $\mathbb K$: $C$ is called an $A_\infty$-algebra, if the coassociative coalgebra $\mathrm T(C[1])$ cofreely cogenerated by $C[1]$ ($[\bullet]$ being the degree-shifting functor on graded vector spaces) with counit admits a coderivation $d_C$ of degree $1$, whose square vanishes.
Similarly, given two $A_\infty$-algebras $(C,d_C)$, $(E,d_E)$ over $\mathbb K$, a graded vector space $M$ over $\mathbb K$ is an $A_\infty$-$C$-$E$-bimodule, if the cofreely cogenerated bi-comodule $\mathrm T(C[1])\otimes M[1]\otimes\mathrm T(E[1])$ with natural left- and right-coactions is endowed with a bi-coderivation $d_M$, whose square vanishes.

Observe that, in view of the cofreeness of $\mathrm T(C[1])$ and $\mathrm T(C[1])\otimes M[1]\otimes\mathrm T(E[1])$, to specify $d_C$, $d_E$ and $d_M$ is equivalent to specify its Taylor components $d_C^n:C[1]^{\otimes n}\to C[1]$, $d_E^n:E[1]^{\otimes n}\to E[1]$, $n\geq 1$, and $d_M^{k,l}:C[1]^{\otimes k}\otimes M[1]\otimes E[1]^{\otimes l}\to M[1]$, $k,l\geq 0$, all of degree $1$: the condition that $d_C$, $d_E$ and $d_M$ square to $0$ is equivalent to an infinite family of quadratic identities between the respective Taylor components.

\section*{Appendix B: on the logarithmic propagator(s)}
Let us review the main results of~\cite{ALRT-1,ALRT-2} for the convenience of the reader by pointing out the main technical details.

Convergence of the integral weights $\varpi^{\log}_\Gamma$, for $\Gamma$ admissible of type $(n,m)$ and $|E(\Gamma)|=2n+m-2$ in the logarithmic case follows from the fact that the integrand $\omega_\Gamma^{\log}$ on $C_{n,m}^+$ extends to a complex-valued, real analytic form of top degree on the compactified configuration space $\overline C_{n,m}^+$.

We must prove that $\omega^{\log}_\Gamma$ extends to all boundary strata of $\overline C_{n,m}^+$: because of the boundary properties of $\omega_{\log}$ ({\em i.e.} $\omega_{\log}$ has a pole of order $1$ along the stratum corresponding to the collapse of its two arguments inside $\mathbb H^+$), the main technical point concerns the extension to boundary strata describing the collapse of clusters of at least two points in $\mathbb H^+$ at different ``speeds'' to single points in $\mathbb H^+$.

By introducing polar coordinates $(\rho_i,\varphi_i)$, $i=1,\dots,k$, for each cluster of collapsing points near such a boundary stratum, the possible poles in $\omega^{\log}_\Gamma$ take the form
\[
\frac{1}{2\pi i}\frac{d\rho_i}{\rho_i}+\frac{d\varphi_i}{2\pi}+\cdots,\ i=1,\dots,k,
\]
where $\cdots$ denotes a complex-valued, real analytic $1$-form.
The angle differential $d\varphi_i$ appears without a factor $\rho_i$ only when paired to the corresponding singular logarithmic differential $d\rho_i/\rho_i$: since $\omega^{\log}_\Gamma$ has top degree and because of skew-symmetry of products of $1$-forms, the singular logarithmic differential $d\rho_i/\rho_i$ must be always paired with $\rho_i d\varphi_i$, coming from the complex-valued, real analytic parts of the factors of $\omega^{\log}_\Gamma$.
The polar coordinates appear naturally by choosing a global section of the trivial principal $G_3=\mathbb R^+\ltimes\mathbb C$-bundle $\mathrm{Conf}_n$ of the configuration space of $n$ points in $\mathbb C$, $n\geq 2$, which identifies it with $S^1\times\mathrm{Conf}_{n-2}(\mathbb C\smallsetminus\{0,1\})$: such a space appears naturally for any cluster of points, the angle coordinates are associated with the $S^1$-factors and the strata are re-covered by setting the radius coordinates to $0$.
The detailed discussion of this topic can be found in the proof of~\cite[Proposition 5.2]{ALRT-2}.

These arguments can be slightly adapted to $\omega^{\log}_\Gamma$, for $\Gamma$ admissible of type $(n,k,l)$ and $|E(\Gamma)|=2n+k+l-1$, where $\omega_{\log}$ is replaced by $\omega^{+,-}_{\log}$: namely, $\omega^{+,-}_{\log}$ on $C_{n,k,l}^+$ extends to a complex-valued, real analytic form of top degree on $\overline C_{n,k,l}^+$.

Similar arguments imply that $\omega^{\log}_\Gamma$, for $\Gamma$ admissible of type $(n,m)$ or $(n,k,l)$ and $|E(\Gamma)|=2n+m-3$ or $|E(\Gamma)|=2n+k+l-2$, yield complex-valued forms on $\overline C_{n,m}^+$ or $\overline C_{n,k,l}^+$ with poles of order $1$ along the boundary.
Moreover, their formal regularizations along boundary strata of codimension $1$ extend to complex-valued, real analytic forms of top degree on those boundary strata: the regularization morphism here formally sets to $0$ the logarithmic differentials $d\rho_i/\rho_i$, whenever $\rho_i=0$.
The detailed version of these arguments can be found in~\cite[Proposition 5.3]{ALRT-2}.
\begin{Thm}\label{t-stokes-reg}
Let $X$ be a compact, oriented manifold with corners of degree $d\geq 2$.
Further, consider an element $\omega$ of $\Omega_1^{d-1}(X)$, which satisfies the two additional properties: 
\begin{itemize} 
\item[$i)$] its exterior derivative $d\omega$ is a complex-valued, real analytic form of top degree on $X$, and
\item[$ii)$] the regularization $\mathrm{Reg}_{\partial X}(\omega)$ along the boundary strata $\partial X$ of codimension $1$ of $X$ is a complex-valued, real analytic form on $\partial X$.
\end{itemize}
Then, the integrals of $d\omega$ over $X$ and the integral of $\mathrm{Reg}_{\partial X}(\omega)$ over $\partial X$ exist and the following identity holds true: 
\begin{equation*}
\int_Xd\omega=\int_{\partial X}\mathrm{Reg}_{\partial X}(\omega).
\end{equation*}
\end{Thm}
In the assumptions of Theorem~\ref{t-stokes-reg}, $\Omega_1^{d-1}(X)$ denotes the space of differential forms $\omega$ on $X$ of degree $d-1$ which have the form
\[
\omega=\sum_{i=1}^p\frac{dx_i}{x_i}\omega_i+\eta
\]
in every local chart of $X$ for which $X=(\mathbb R_+)^p\times \mathbb R^q$, $p+q=d$, and $\omega_i$, $\eta$, $i=1,\dots,p$, are complex-valued, real analytic forms on $X$.
The proof of Theorem~\ref{t-stokes-reg}, as well as of other variants of Stokes' Theorem in presence of singularities, can be found in~\cite[Subsection 2.3]{ALRT-1}.

Since $\omega^{\log}_\Gamma$ is closed and because of the previous arguments, Stokes' Theorem~\ref{t-stokes-reg} applies to $\omega^{\log}_\Gamma$, whence the associativity of $\star_{\log}$ (more generally, the $L_\infty$-relations for the logarithmic formality quasi-isomorphism and its corresponding version in presence of two branes).

\end{document}